\documentclass[12pt, a4paper]{amsart} 

\usepackage{mathptmx} 
\usepackage{amsfonts,amsthm,amssymb,dsfont,stmaryrd,mathrsfs} 
\usepackage[leqno]{amsmath} 
\usepackage{array}
\usepackage{colortbl}
\usepackage{comment} 

\usepackage{graphicx,caption} 
\usepackage[pdfstartview=FitH]{hyperref} 
\usepackage{pdfpages} 
\usepackage{longtable,supertabular, multirow} 
\usepackage{doi}
\usepackage{tikz, tikz-cd}
\usetikzlibrary{matrix,arrows,decorations.pathmorphing}

\usepackage[hmarginratio={1:1}, vmarginratio={1:1}, top=3cm, bottom=2cm, left=3cm, right=3cm]{geometry} 

\usepackage{setspace}
\usepackage{mathtools} 

\theoremstyle{plain}
\newtheorem*{thm*}{Theorem} 
\newtheorem{thm}{Theorem} [section]
\newtheorem*{thm2*}{Theorem \ref{thmDecNilp}}
\newtheorem*{thm3*}{Theorem \ref{thmDecNilp2}}

\newtheorem{thmA}{Theorem}
\newtheorem{lem}[thm]{Lemma}
\newtheorem*{lem*}{Lemma}
\newtheorem{prop}[thm]{Proposition}
\newtheorem*{prop*}{Proposition}

\newtheorem{cor}[thm]{Corollary}
\newtheorem*{cor*}{Corollary}

\theoremstyle{definition}

\newtheorem*{defn*}{Definition}

\newtheorem{conjectureA}{Conjecture}

\newtheorem*{conjecture*}{Conjecture}
\newtheorem{exmp}[thm]{Example}
\newtheorem*{exmp*}{Example}

\newtheorem*{prob*}{Problem}

\newtheorem*{ques*}{Question}

\newtheorem*{remk*}{Remark}

\def\nn{\mathbb{N}}
\def\zz{\mathbb{Z}}

\begin{document}

\title{Conflict-Avoiding Codes of Prime Lengths and Cyclotomic Numbers}

\author{Liang-Chung Hsia}
\address{Department of Mathematics, National Taiwan Normal University, Taipei 11677, Taiwan, ROC}
\email{hsia@math.ntnu.edu.tw}

\author{Hua-Chieh Li}
\address{Department of Mathematics, National Taiwan Normal University, Taipei 11677, Taiwan, ROC}
\email{li@math.ntnu.edu.tw}

\author{Wei-Liang Sun}
\address{Department of Mathematics, National Taiwan Normal University, Taipei 11677, Taiwan, ROC}
\email{wlsun@ntnu.edu.tw}

\maketitle

\begin{abstract}
The problem to construct optimal conflict-avoiding codes of even lengths and the Hamming weight $3$ is completely settled. On the contrary, it is still open for odd lengths. It turns out that the prime lengths are the fundamental cases needed to be constructed. 
In the article, we study conflict-avoiding codes of prime lengths and give a connection with the so-called cyclotomic numbers. By having some nonzero cyclotomic numbers, a well-known algorithm for constructing optimal conflict-avoiding codes will work for certain prime lengths. As a consequence, we are able to answer the size of optimal conflict-avoiding code for a new class of prime lengths. \\

\end{abstract}

\section{Introduction}
\label{sec1}

A binary protocol sequence set for transmitting data packets over a multiple-access collision channel without feedback is called a {\it conflict-avoiding code} (CAC) in information theory and has been studied few decades ago by \cite{Mat, NguGyoMas, GyoVaj, TsyRub, LevTon, Lev}. A mathematical model for CACs of {\it length} $n$ and ({\it Hamming}) {\it weight} $w$ is as follows. Let $\zz_n = \zz / n \zz$ be the additive group of $\zz$ modulo $n$. For a $w$-subset ${\bf x} = \{x_1, \ldots, x_w\}$ of $\zz_n$, denote the difference set of ${\bf x}$ by $\Delta({\bf x}) = \{x_i - x_j \mid i \neq j\}$. A CAC of length $n$ and weight $w$ is a collection $\mathcal{C}$ of $w$-subsets of $\zz_n$ such that $\Delta({\bf x}) \cap \Delta({\bf y}) = \varnothing$ for every distinct ${\bf x}, {\bf y}$ of $\mathcal{C}$. Each ${\bf x}$ is called a {\it codeword}. A CAC is said to be {\it optimal} if it has the maximal size among all CACs of the same length and weight. When the weight is $1$ or $2$, there is no difficulty to find the optimal size. For weights more than $2$, however, finding an optimal CAC and determining its size is still an open problem. The first challenge is the weight $3$ and we introduce its development below. 

Assume that the weight $w = 3$. The construction of an optimal CAC of an arbitrary even length is completely found by many authors \cite{LevTon, JimMisJanTeyTon, MisFuUru, FuLinMis}. However, the known results for odd lengths are still far from completion. In the case that the length $n$ is odd, the multiplicative order of $2$ in the multiplicative group of units of $\zz_p$ for prime divisors $p$ of $n$ becomes a key role in constructing an optimal CAC. For convenience, we denote $\zz_p^{\times}$ by the multiplicative group of $\zz_p$ and $o_p(2)$ by the multiplicative order of $2$ in $\zz_p^{\times}$ for an odd prime $p$. 
A construction of optimal CAC is found in \cite{LevTon} for prime length $p$ such that $4 \mid o_p(2)$. Later, the construction is improved in \cite{Lev} for composite number length where every prime divisor $p$ of the length satisfies $4 \mid o_p(2)$. 

The remaining situation is that the length has a prime divisor $p$ such that $4 \nmid o_p(2)$. In \cite{FuLoShu}, the authors prove that the problem of such length can be reduced to the problem of length $n$ such that {\it every} prime divisor $p$ of $n$ satisfies $4 \nmid o_p(2)$. They also prove that an optimal CAC of length $n = p^k$ can be constructed from an optimal CAC of length $p$ where $p$ is a non-Wieferich prime (that $2^{p-1} \not\equiv 1 \pmod{p^2}$). However, the construction for prime length $p$ is not completely found. For other odd lengths 
one can refer to \cite{Mom, WuFu, LinMisSatJim, MisMom, HsiLiSun} for the constructions. It turns out that the prime lengths are fundamental objects needed to be solved. This leads us to study CACs prime lengths and weight $3$.


For an odd prime length $p > 3$, a difference set $\Delta({\bf x})$ would have size $4$ or $6$ for a $3$-subset ${\bf x}$ of $\zz_p$. Without loss of generality, we can assume the codeword ${\bf x}$ to be $\{0, \alpha, \beta\}$ for some $\alpha, \beta \in \zz_p^{\times}$. Then $|\Delta({\bf x})| = 4$ if and only if ${\bf x} = \{0, \alpha, 2 \alpha\}$ for some $\alpha$. In this case, ${\bf x}$ is called an {\it equi-difference} codeword. To have an optimal CAC, it is natural to have equi-difference codewords as many as possible. 
If $M^e(p)$ denotes the maximal size among all CACs consisting of equi-difference codewords only, then one has 
$$M^e(p) = \frac{p-1-2 O(p)}{4}$$ where $O(p) = 0$ if $4$ divides $\left| \langle -1, 2 \rangle \right|$ and $O(p) = (p-1) / \left| \langle -1, 2 \rangle \right|$ otherwise. 
Moreover, \cite[Lemma~3]{FuLoShu} provides
\begin{align}
\label{upper bound: FLS}
M^e(p) \leq M(p) \leq M^e(p) + \left\lfloor \frac{O(p)}{3} \right\rfloor
\end{align}
where $M(p)$ denotes the maximal size among all CACs. 
Thus, if $O(p) < 3$, then a CAC consisting of equi-difference codewords only must be optimal. 
It remains to consider $O(p) \geq 3$. 
The authors of \cite{FuLoShu} propose an algorithm for constructing nonequi-difference codewords. They conjecture that the algorithm always works for prime length $p$ for constructing $\left\lfloor \frac{O(p)}{3} \right\rfloor$ nonequi-difference codewords. Then $\frac{p-1-2 O(p)}{4}$ equi-difference codewords could be further obtained and this CAC would be optimal by the upper bound in \eqref{upper bound: FLS}. 
For our purpose, we rephrase the statement of their conjecture in terms of cosets of the subgroup $\langle -1, 2 \rangle$, generated by $-1$ and $2$, in the multiplicative group $\zz_p^{\times} = \zz_p \setminus \{0\}$. 

\begin{conjectureA}[{\cite[Conjecture~1]{FuLoShu}}]
\label{conjA}
Let $p$ be a non-Wieferich prime. 
Then there are $3 \left\lfloor \frac{O(p)}{3} \right \rfloor$ cosets 
$A_1, B_1, C_1, \ldots, 
A_{\left\lfloor \frac{O(p)}{3} \right \rfloor}, 
B_{\left\lfloor \frac{O(p)}{3} \right \rfloor}, 
C_{\left\lfloor \frac{O(p)}{3} \right \rfloor}
$ 
of the subgroup $\langle -1, 2 \rangle$ in 
$\zz_p^{\times}$ such that for each $i$ there exists $(a_i, b_i, c_i) \in A_i \times B_i \times C_i$ satisfying 
$$
a_i + b_i = c_i \quad \text{in } \zz_p.
$$
\end{conjectureA}

We remark that each $(a_i, b_i, c_i)$ corresponds to the nonequi-difference codeword 
${\bf x}_i = \{0, a_i, c_i\}$ 
where 
$\Delta({\bf x}_i) = \{\pm a_i, \pm b_i, \pm c_i\}$. 
Those triples $(a_i, b_i, c_i)$ above correspond to a {\it hypergraph matching} in an auxiliary hypergraph in their terminology. 
We give an example to illustrate how Conjecture~\ref{conjA} and their algorithm work for constructing an optimal CAC. 

Consider CACs of length $p = 73$ and weight $3$ where $O(p) = 4$. First, there are four cosets of $\langle -1, 2 \rangle$ in $\zz_{73}^{\times}$: $$\langle -1, 2 \rangle, \quad 3 \langle -1, 2 \rangle, \quad 5 \langle -1, 2 \rangle, \quad 11 \langle -1, 2 \rangle.$$ 
We pick elements $a, b, c$ in three different cosets respectively such that $a + b = c$. For example, we choose $a = 2$, $b = 3$ and $c = 5$ from the first three cosets. Let ${\bf x} = \{0, 2, 5\}$. Then ${\bf x}$ is a nonequi-difference codeword and $\Delta({\bf x}) = \{\pm 2, \pm 3, \pm 5\}$. The elements in $\langle -1, 2 \rangle \setminus \Delta({\bf x})$ lead to four equi-difference codewords $${\bf x}_1 = \{0, 2^2, 2^3\}, \quad {\bf x}_2 = \{0, 2^4, 2^5\}, \quad {\bf x}_3 = \{0, 2^6, 2^7\}, \quad {\bf x}_4 = \{0, 2^8, 2^9\}.$$ It is similar to $3 \langle -1, 2 \rangle \setminus \Delta({\bf x})$, $5 \langle -1, 2 \rangle \setminus \Delta({\bf x})$ and $11 \langle -1, 2 \rangle \setminus \Delta({\bf x})$ that $${\bf x}_{i+4} = \{0, 3 \cdot 2^{2i}, 3 \cdot 2^{2i+1}\}, \quad {\bf x}_{i+8} = \{0, 5 \cdot 2^{2i}, 5\cdot 2^{2i+1}\}, \quad  {\bf x}_{i+12} = \{0, 11 \cdot 2^{2i}, 11 \cdot 2^{2i+1}\},$$ respectively, for $i = 1, 2, 3, 4$. The difference sets of all codewords ${\bf x}, {\bf x}_1, \ldots, {\bf x}_{16}$ are pairwisely disjoint. 
By the upper bound of $M(p)$ in \eqref{upper bound: FLS}, the collection $\mathcal{C} = \{{\bf x}, {\bf x}_1, \ldots, {\bf x}_{16}\}$ forms an optimal CAC which consists of $1$ nonequi-difference codeword and $16$ equi-difference codewords. If we choose $a = 6$, $b = 5$ and $c = 11$ such that $a + b = c$, then ${\bf x} = \{0, 6, 11\}$ and $\Delta({\bf x}) = \{\pm 6, \pm 5, \pm 11\}$. The same procedure will lead to another $16$ equi-difference codewords. Different choices of $a, b, c$ will give different optimal CACs.

The authors of \cite{MaZhaShe} independently give a similar thought to the existence of the triples $(A_i, B_i, C_i)$ in Conjecture~\ref{conjA} in terms of the group structure of the quotient group $\zz_p^{\times} / \langle -1, 2 \rangle$. 
They notice that $\zz_p^{\times} / \langle -1, 2 \rangle$ is a cyclic group because $\zz_p^{\times}$ is cyclic. 

\begin{conjectureA}[{\cite[Conjecture]{MaZhaShe}}]
\label{conjB}
Let $p$ be an odd prime. If $\left[ \zz_p^{\times} : \langle -1, 2 \rangle \right] \geq 3$, then there is a generator $t \langle -1, 2 \rangle$ of the cyclic group $\zz_p^{\times} / \langle -1, 2 \rangle$ for some $t \in \zz_p^{\times}$ such that $$1 + b = c$$ for some $b \in t \langle -1, 2 \rangle$ and $c \in t^2 \langle -1, 2 \rangle$.
\end{conjectureA}

Note that $\left[ \zz_p^{\times} : \langle -1, 2 \rangle \right] \geq 3$ implies either $O(p) = 0$ or $O(p) \geq 3$. Moreover, if Conjecture~\ref{conjB} holds for a prime $p$, then one can set $A_1 = t^0 \langle -1, 2 \rangle$, $B_1 = t^1 \langle -1, 2 \rangle$, $C_1 = t^2 \langle -1, 2 \rangle$, $A_2 = t^3 \langle -1, 2 \rangle$, $B_2 = t^4 \langle -1, 2 \rangle$, $C_2 = t^5 \langle -1, 2 \rangle$ and so on to obtain $t^{3 i} + t^{3 i} b = t^{3 i} c$ for each $i$ where $(t^{3i}, t^{3 i} b, t^{3 i} c) \in A_i \times B_i \times C_i$. Hence, Conjecture~\ref{conjB} implies Conjecture~\ref{conjA}. 
For the previous example that $p = 73$, $t$ can be chosen as $5$. 
Then $b = 5$, $c = 6 \in 5^2 \langle -1, 2 \rangle$ and $1 + 5 = 6$ in $\zz_p$. 
Moreover, $p$ is not assumed to be non-Wieferich. 
In \cite{MaZhaShe}, the authors have checked in computer that Conjecture~\ref{conjB} holds for prime $p \leq 2^{30}$. 

The motivation of this article is to study Conjecture~\ref{conjB}. Throughout this article, we denote $$\ell = \left[ \zz_p^{\times} : \langle -1, 2 \rangle \right] =\left|\zz_p^{\times} / \langle -1, 2 \rangle\right|$$ for a given prime $p$. 
The existence of $b$ and $c$ in Conjecture~\ref{conjB} is strongly related to the so-called {\it cyclotomic numbers $A(i, j)$ of order $\ell$} with respect to a primitive root where $i, j$ are integers. In particular, both $b$ and $c$ exist if and only if $A(i, 2 i) \neq 0$ for certain $i$. 
More precisely, Conjecture~\ref{conjB} is equivalent to the following statement. See Section~\ref{sec2} for detail.

\begin{conjecture*}
Let $p$ be an odd prime and let $\ell = \left[ \zz_p^{\times} : \langle -1, 2 \rangle \right]$. If $\ell \geq 3$, then $A(i, 2i) \neq 0$ for some $1 \leq i \leq \ell-1$ with $(i, \ell) = 1$. 
\end{conjecture*}

The cyclotomic numbers was considered by C.F. Gauss \cite{Gau} and later L.E. Dickson \cite{Dic, Dic3, Dic4} studied these numbers in the context of Waring's problem. One of the central problems is to determine these numbers in terms of solutions of a certain Diophantine systems. 
See also the surveys \cite{Raj,Kat2, AhmTan}. However, it is still not obvious whether $A(i, 2i) \neq 0$ because it is difficult either to find a suitable Diophantine system or to write down the formulas of cyclotomic numbers when $\ell$ is getting larger. Recently, cyclotomic numbers are developed into a secured public-key cryptography model \cite{AhmTanPus}.

In this article, we study the conjecture above and we will prove the following result.

\begin{thmA}
\label{thmA}
Let $p$ be an odd prime and let $\ell = [\zz_p^{\times} : \langle -1, 2 \rangle]$. If $\ell$ is an odd prime, then $A(i, 2 i) \neq 0$ for some $1 \leq i \leq \ell-1$ with $(i, \ell) = 1$.
\end{thmA}

This not only answers Conjecture~\ref{conjB} but Conjecture~\ref{conjA} for a class of primes $p$. Moreover, the size $M(p)$ of optimal CAC can be determined as follows, while $O(p) = \ell$ if $4 \nmid o_p(2)$.

\begin{thmA}
\label{thmB}
Let $p$ be an odd prime with $4 \nmid o_p(2)$. If $\ell = [\zz_p^{\times} : \langle -1, 2 \rangle]$ is an odd prime, then an optimal conflict-avoiding code of length $p$ and weight $3$ has the size $$\frac{p-1-2 \ell}{4} + \left\lfloor \frac{\ell}{3} \right\rfloor.$$
\end{thmA} 


This article is organized as follows. In Section~\ref{sec2}, we introduce cyclotomic numbers and see how these numbers connect to Conjecture~\ref{conjB}. 
The cyclotomic numbers will lead to an interesting matrix in Section~\ref{sec3}. This matrix can be used to answer Conjecture~\ref{conjB}, as well as Conjecture~\ref{conjA}, for primes with small $\ell$. 
In Section~\ref{sec4}, we will see that the number of squares in a certain subset of $\zz_p$ can be represented to a particular sum of cyclotomic numbers. By counting squares and using an upper bound of cyclotomic numbers given in \cite{BetHirKomMun}, we will be able to prove Theorem~\ref{thmA}. 
Finally, we remark how an obstacle is encountered for primes with composite number $\ell$.

\section{Preliminaries}
\label{sec2}

In this article, $p$ is an odd prime and $\zz_p$ is the finite field of $p$ elements. Throughout this article, we denote $$L = \langle -1, 2 \rangle$$ the subgroup of $\zz_p^{\times} = \zz_p \setminus \{0\}$ generated by $-1$ and $2$, and then $\ell = \left[ \zz_p^{\times} : L \right]$ is the index of $L$ in $\zz_p^{\times}$. We always assume $$\ell \geq 3.$$ A generator of $\zz_p^{\times}$ is also called a {\it primitive root} of $\zz_p$. Fix a primitive root $g$. 
Define the number $$A(i, j) = \left| \left\{(m, n) \mid 0 \leq m, n \leq \frac{p-1}{\ell} - 1 \text{ and } 1 + g^{i + m \ell} = g^{j + n \ell} \right\} \right|$$ for $i, j \in \zz$.
These numbers $A(i, j)$ are called {\it cyclotomic numbers of order $\ell$ with respect to $g$}. We can rewrite $A(i, j)$ as follows. Note that $|L| = \frac{p-1}{\ell}$ and $L$ consists of all $\ell$-th power of elements of $\zz_p^{\times}$. Thus, both $g^{m \ell}, g^{n \ell} \in L$ for $0 \leq m, n \leq |L| - 1$. 
It follows that $$A(i, j) = \left| (1 + g^i L) \cap g^j L \right|$$ for $i, j \in \zz$. If $i'$ and $j'$ are integers such that $i' \equiv i \pmod{\ell}$ and $j' \equiv j \pmod{\ell}$, then $g^{i'}L = g^{i} L$ and $g^{j'} L = g^{j} L$. We have $A(i', j') = A(i, j)$. Hence, it is convenient to consider $i, j$ modulo $\ell$.

As $g$ is a primitive root and $\ell = \left[ \zz_p^{\times} : L \right]$, the quotient group $\zz_p^{\times} / L$ consists of elements $g^0 L = L, g^1 L, \ldots, g^{\ell-1} L$. In particular, the set $$\left\{g^i L \mid 1 \leq i \leq \ell \text{ and } (i, \ell) = 1\right\}$$ is a complete set of generators of $\zz_p^{\times} / L$ where $(\cdot, \cdot)$ denotes the greatest common divisor of two given integers. Let $t \in \zz_p^{\times}$. Then $t L$ generates $\zz_p^{\times} / L$ if and only if $t L = g^i L$ for some $1 \leq i \leq \ell$ with $(i, \ell) = 1$. 
Moreover, $$\left| (1 + t L) \cap t^2 L \right| = \left| (1 + g^i L) \cap g^{2 i} L \right| = A(i, 2 i).$$ Thus, Conjecture~\ref{conjB} holds for $p$ if and only if $A(i, 2 i) \neq 0$ for some $1 \leq i \leq \ell$ with $(i, \ell) = 1$. Because each $A(i, 2i)$ is nonnegative, it is enough to consider the following sum $$s(\ell) = \sum_{\substack{1 \leq i \leq \ell, \\ (i, \ell) = 1}} A(i, 2i).$$ The argument above actually shows the following lemma.

\begin{lem}
\label{lem2.3}
Let $p$ and $\ell$ be as above. Then $s(\ell) \neq 0$ if and only if \textup{Conjecture~\ref{conjB}} holds for this $p$.
\end{lem}

The sum $s(\ell)$ is defined with respect to a given primitive root $g$. In fact, we can freely compute $s(\ell)$ when picking an {\it arbitrary} primitive root. This helps a lot in practice. We leave a proof below for the readers because we do not find any reference for this fact. 

\begin{prop}
\label{prop2.1}
The sum $s(\ell)$ is independent on the choices of primitive roots. 
\end{prop}

\begin{proof}
Let $h$ be a primitive root of $\zz_p$. Denote $A'(i, j) = |(1 + h^i L) \cap h^j L|$ and $s'(\ell)$ the sum of $A'\left(i, 2i\right)$ for $i = 1, \ldots, \ell$ with $(i, \ell) = 1$. Since $h L$ also generates $\zz_p^{\times} / L$, $h L = g^k L$ for some $1 \leq k \leq \ell$ with $(k, \ell) = 1$. Then $h^i L = g^{k i} L$ and thus $A'(i, j) = A(k i, k j)$ for $i, j \in \zz$. Recall that $\{1 \leq i \leq \ell \mid (i, \ell) = 1\}$ is a reduced residue system modulo $\ell$. Since $(k, \ell) = 1$, it follows that $\{k i \mid 1 \leq i \leq \ell \text{ and } (i, \ell) = 1\}$ is also a reduced residue system modulo $\ell$. Thus, $$s'(\ell) = \sum_{\substack{1 \leq i \leq \ell, \\ (i, \ell) = 1}} A'(i, 2i) = \sum_{\substack{1 \leq i \leq \ell, \\ (i, \ell) = 1}} A\left(k i, k(2i)\right) = \sum_{\substack{1 \leq i \leq \ell, \\ (i, \ell) = 1}} A(i, 2i) = s(\ell).$$
\end{proof}

As we have mentioned, one of main problems on the theory of cyclotomic numbers is to find explicit formulas for all $A(i, j)$ in terms of solutions of certain Diophantine systems. For example, consider $p \equiv 1 \pmod{3}$ such that $\ell = 3$. It is well-known that the Diophantine system $$\left\{ \begin{array}{c} 4 p = X^2 + 27 Y^2, \\ X \equiv 1 \pmod{3} \end{array} \right.$$ has a solution $(a, b)$ with unique $a$, where $-p < a < p$. \cite[Section~358]{Gau} gives that $$A(0,0) = \frac{1}{9} (p - 8 + a),$$ $$ A(0,1) = A(1,0) = A(2,2) = \frac{1}{18} (2p - 4 - a + 9 b),$$ $$A(0,2) = A(1,1) = A(2, 0) = \frac{1}{18} (2p - 4 - a - 9 b),$$ $$A(1, 2) = A(2, 1) = \frac{1}{9} (p + 1 + a)$$ where the sign of $b$ depends of the choices of primitive roots. In this case, $A(1, 2) = A(2, 1) > 0$ and thus $$s(3) = A(1,2) + A(2, 4) = A(1,2) + A(2,1) >0,$$ while $A(2, 4) = A(2, 1)$ because $4 \equiv 1 \pmod{\ell}$. 
However, the formulas for large $\ell$ will be extremely complicated and hard to find. For instance, the formulas for $\ell = 11$ are presented in \cite{LeoWil}. But it is not obvious whether $s(11) \neq 0$ or not for $p \equiv 1 \pmod{11}$. See also Example~\ref{exmp4.1}. \\

We deduce the following useful properties. 

\begin{lem}
\label{lem4.1}
For $i, j$ modulo $\ell$, $$A(i, j) = A(-i, j-i) \quad \text{and} \quad A(i, j) = A(j, i).$$
\end{lem}

\begin{proof}
For $\alpha \in (1 + g^i L) \cap g^j L$, write $\alpha = 1 + g^i a = g^j b$ for some $a, b \in L$. One has $1 + (\alpha - 1)^{-1} = 1 + g^{-i} a^{-1} \in 1 + g^{-i} L$. Moreover, $1 + (\alpha - 1)^{-1} = (g^i a + 1) g^{-i} a^{-1} = (g^{j} b) (g^{-i} a^{-1}) = g^{j-i} (b a^{-1}) \in g^{j-i} L$. This gives a map $\varphi$ from $(1 + g^i L) \cap g^j L$ into $(1 + g^{-i} L) \cap g^{j-i} L$ given by $\varphi(\alpha) = 1 + (\alpha - 1)^{-1}$. One can deduce that $\varphi$ is a bijection. 
Hence, the first equality holds.  

For the second equality, observe that $1 - \alpha = 1 + g^j (-b) = g^i (-a) \in (1 + g^j L) \cap g^i L$ as $-1 \in L$. Then $\alpha \mapsto 1 - \alpha$ gives a bijection from $(1 + g^i L) \cap g^j L$ onto $(1 + g^j L) \cap g^i L$. 
\end{proof}

\section{Extended Cyclotomic Matrices}
\label{sec3}

The modulo relation and properties of cyclotomic numbers lead us to an interesting matrix. This will provide a new insight to the conflict-avoiding code problem. In this section, we will be able to easily answer Conjecture~\ref{conjB} for small $\ell$.

For a given primitive root of $\zz_p$, the {\it cyclotomic matrix} is an $\ell$-by-$\ell$ matrix defined as follows $$(A(i, j))_{\ell \times \ell} \quad \text{for } i, j = 0, 1, \ldots, \ell-1.$$ 
It is convenient to give a variation of the cyclotomic matrix. We define the {\it extended cyclotomic matrix} $B = (b_{i, j})_{\ell \times \ell}$ as follows $$b_{i, j} = \left\{ \begin{array}{ll} A(0, 0) + 1 & \text{if $i = j = 0$;} \\ A(i, j) & \text{otherwise}. \end{array} \right.$$ 
For $i, j$ modulo $\ell$, Lemma~\ref{lem4.1} gives the relation for entries: \begin{align} b_{i, j} = b_{-i, j-i} = b_{j-i, -i} = b_{i-j, -j} = b_{-j, i-j} = b_{j, i}.\label{eq5}\end{align} 
For convenience, the row $(b_{i,0}, b_{i, 1}, \ldots, b_{i, \ell-1})$ of $B$ is said to be {\it indexed by $i$} for $i = 0, \ldots, \ell-1$. 
This matrix $B$ has following nice properties.

\begin{lem}
\label{lem3.3}
Let $B$ be the extended cyclotomic matrix as above. 
\begin{enumerate}
\item[(i)]
$B$ is symmetric with nonnegative integer entries and $b_{0,0} > 0$.
\item[(ii)] 
The diagonal entries after $b_{0,0}$ is the reverse of the row indexed by $0$\textup{:} $$(b_{1,1}, b_{2,2}, \ldots, b_{\ell-1, \ell-1}) = (b_{0, \ell-1}, b_{0, \ell-2}, \ldots, b_{0, 1}).$$
\item[(iii)]
For $1 \leq i \leq \ell-1$, the row indexed by $\ell - i$ is shifted $i$ times to the left by the row indexed by $i$\textup{:} $$(b_{\ell-i, 0}, b_{\ell-i, 1}, \ldots, b_{\ell-i, \ell-1}) = (b_{i, i}, b_{i, i+1}, \ldots, b_{i, \ell-1}, b_{i, 0}, \ldots, b_{i, i-1})$$
\item[(iv)]
The trace, the sum of each row and the sum of each column are all $|L| = \frac{p-1}{\ell}$.
\end{enumerate}
\end{lem}

\begin{proof}
(i), (ii) and (iii) are obvious from (\ref{eq5}) that $b_{i, j} = b_{j, i}$, $b_{i, i} = b_{0, \ell-i}$ and $b_{\ell - i, j} = b_{i, j + i}$.

For (iv), let $g$ be the given primitive root. Then the group $\zz_p^{\times}$ is the disjoint union of $g^0 L, g^1 L, \ldots, g^{\ell-1} L$. Since $1 \not\in 1 + L$, the sum of row indexed by $0$ is $$\sum_{j = 0}^{\ell-1} b_{0,j} = 1 + \sum_{j = 0}^{\ell-1} A(0, j) = 1 + \sum_{j=0}^{\ell-1} \left|(1 + L) \cap g^j L \right| = 1 + \left| (1 + L) \cap \zz_p^{\times} \right| = |L|.$$ Moreover, the sum of row indexed by $i$ for $i \neq 0$ is $$\sum_{j=0}^{\ell-1} b_{i, j} = \sum_{j=0}^{\ell-1} A(i, j) = \sum_{j=0}^{\ell-1} \left|(1 + g^i L) \cap g^j L \right| = \left| (1 + g^i L) \cap \zz_p^{\times} \right| = |L|.$$ The trace of $B$ is also $|L|$ by (ii). Finally, the sum of each column is $|L|$ by (i).
\end{proof}

Recall the sum $s(\ell)$ and we have $$s(\ell) = \sum_{\substack{1 \leq i \leq \ell, \\ (i, \ell) = 1}} b_{i, 2i}.$$ The relations~(\ref{eq5}) and Lemma~\ref{lem3.3} lead us to answer that $s(\ell) > 0$ for small $\ell$.

\begin{prop}
\label{prop3.1}
For $\ell = 3, 4, 5$, we have $s(\ell) > 0$.
\end{prop}

\begin{proof}
Using (\ref{eq5}) and (i), (ii), (iii) of Lemma~\ref{lem3.3}, we make the matrix $B$ less complicated. 
For convenience, we denote $r_i$ by the row indexed by $i$.

For $\ell = 3$, $$B = \left( \begin{matrix} b_{0,0} & b_{0,1} & b_{0,2} \\ b_{1,0} & b_{1,1} & b_{1,2} \\ b_{2,0} & b_{2,1} & b_{2,2} \end{matrix} \right).$$ By symmetry, $b_{1,0} = b_{0,1}$, $b_{2,0} = b_{0,2}$ and $b_{2,1} = b_{1,2}$. By Lemma~\ref{lem3.3}(ii), $(b_{1,1}, b_{2,2}) = (b_{0,2}, b_{0,1})$. Thus, $$B = \left( \begin{matrix} b_{0,0} & b_{0,1} & b_{0,2} \\ b_{0,1} & b_{0,2} & b_{1,2} \\ b_{0,2} & b_{1,2} & b_{0,1} \end{matrix} \right).$$ By Lemma~\ref{lem3.3}(iv), the sums of $r_0$ and $r_1$ are equal. We have $b_{0,0} = b_{1,2}$. We obtain $s(3) = b_{1,2} + b_{2,4} \geq b_{1,2} = b_{0,0} > 0.$

For $\ell = 4$, $$B = \left( \begin{matrix} b_{0,0} & b_{0,1} & b_{0,2} & b_{0,3} \\ b_{1,0} & b_{1,1} & b_{1,2} & b_{1,3} \\ b_{2,0} & b_{2,1} & b_{2,2} & b_{2,3} \\ b_{3,0} & b_{3,1} & b_{3,2} & b_{3,3} \end{matrix} \right).$$ By Lemma~\ref{lem3.3}(ii), $b_{1,1} = b_{0,3}$. By (\ref{eq5}), we can deduce that $b_{1,3} = b_{1,2}$. So $r_1$ becomes $(b_{1,0}, b_{0,3}, b_{1,2}, b_{1,2})$. According to Lemma~\ref{lem3.3}(iii), $r_2$ satisfies $(b_{2,0}, b_{2,1}, b_{2,2}, b_{2,3}) = (b_{2,2}, b_{2,3}, b_{2,0}, b_{2,1})$ and then it is $(b_{2,0}, b_{2,1}, b_{2,0}, b_{2,1})$. 
Moreover, $r_3 = (b_{0,3}, b_{1,2}, b_{1,2}, b_{1,0})$ by a shifting of $r_1$. Thus, by symmetry, we have 
$$B = \left( \begin{matrix} b_{0,0} & b_{0,1} & b_{0,2} & b_{0,3} \\ b_{0,1} & b_{0,3} & b_{1,2} & b_{1,2} \\ b_{0,2} & b_{1,2} & b_{0,2} & b_{1,2} \\ b_{0,3} & b_{1,2} & b_{1,2} & b_{0,1} \end{matrix} \right).$$ 
Suppose that $s(4) = b_{1,2} + b_{3, 2} = 0$. Then $b_{1,2} = 0$ and $$B = \left( \begin{matrix} b_{0,0} & b_{0,1} & b_{0,2} & b_{0,3} \\ b_{0,1} & b_{0,3} & 0 & 0 \\ b_{0,2} & 0 & b_{0,2} & 0 \\ b_{0,3} & 0 & 0 & b_{0,1} \end{matrix} \right).$$ By Lemma~\ref{lem3.3}(iv), the sums of $r_0$ and $r_1$ are equal. We obtain $b_{0,0} = 0$. This contradicts to $b_{0,0} > 0$. Thus, $s(4) \neq 0$.

For $\ell = 5$, $$B = \left( \begin{matrix} b_{0,0} & b_{0,1} & b_{0,2} & b_{0,3} & b_{0,4} \\ b_{1,0} & b_{1,1} & b_{1,2} & b_{1,3} & b_{1,4} \\ b_{2,0} & b_{2,1} & b_{2,2} & b_{2,3} & b_{2,4} \\ b_{3,0} & b_{3,1} & b_{3,2} & b_{3,3} & b_{3,4} \\ b_{4,0} & b_{4,1} & b_{4,2} & b_{4,3} & b_{4,4} \end{matrix} \right).$$ By Lemma~\ref{lem3.3}(ii), $b_{1,1} = b_{0,4}$ and $b_{2,2} = b_{0,3}$. By (\ref{eq5}), we can deduce that $b_{1,4} = b_{1,2}$ and $b_{2,3} = b_{2,4} = b_{1,3}$. 
Thus, $r_1 = (b_{1,0}, b_{0,4}, b_{1,2}, b_{1,3}, b_{1,2})$ and $r_2 = (b_{2,0}, b_{2,1}, b_{0,3}, b_{1,3}, b_{1,3})$. By Lemma~\ref{lem3.3}(iii), $r_3 = (b_{0,3}, b_{1,3}, b_{1,3}, b_{2,0}, b_{2,1}$ and $r_4 = (b_{0,4}, b_{1,2}, b_{1,3}, b_{1,2}, b_{1,0})$ by shifting $r_2$ and $r_1$, respectively. Hence, by symmetry, we have $$B = \left( \begin{matrix} b_{0,0} & b_{0,1} & b_{0,2} & b_{0,3} & b_{0,4} \\ b_{0,1} & b_{0,4} & b_{1,2} & b_{1,3} & b_{1,2} \\ b_{0,2} & b_{1,2} & b_{0,3} & b_{1,3} & b_{1,3} \\ b_{0,3} & b_{1,3} & b_{1,3} & b_{0,2} & b_{1,2} \\ b_{0,4} & b_{1,2} & b_{1,3} & b_{1,2} & b_{0,1} \end{matrix} \right).$$ Suppose that $s(5) = b_{1,2} + b_{2,4} + b_{3,1} + b_{4,3} = 0$. Then $b_{1,2} = b_{1,3} = 0$ and we have $$B = \left( \begin{matrix} b_{0,0} & b_{0,1} & b_{0,2} & b_{0,3} & b_{0,4} \\ b_{0,1} & b_{0,4} & 0 & 0 & 0 \\ b_{0,2} & 0 & b_{0,3} & 0 & 0 \\ b_{0,3} & 0 & 0 & b_{0,2} & 0 \\ b_{0,4} & 0 & 0 & 0 & b_{01} \end{matrix} \right).$$ By Lemma~\ref{lem3.3}(iv), the sums of $r_0$ and $r_1$ are equal. We obtain $b_{0,0} = 0$, a contradiction. Thus, $s(5) \neq 0$.
\end{proof}

According to Lemma~\ref{lem2.3}, a prime $p$ will satisfy Conjecture~\ref{conjB}, as well as Conjecture~\ref{conjA}, if $s(\ell) \neq 0$. Thus we have the following conclusion on the size of optimal CAC.

\begin{cor}
\label{cor3.9}
Let $p$ be an odd prime such that $4 \nmid o_p(2)$. If $\ell = \left[\zz_p^{\times} : \langle -1, 2 \rangle\right] = 3, 4, 5$, then 
the size of optimal conflict-avoiding codes of length $p$ and weight $3$ is $\frac{p-1-2\ell}{4} + 1$.
\end{cor}

Corollary~\ref{cor3.9} is also known in \cite[Corollary~4.8]{MaZhaShe} in a different method. 
When $\ell \geq 6$, the above approach is not enough to check whether $s(\ell) \neq 0$. In the next section, we discover more properties on $s(\ell)$.

\section{The number of squares}
\label{sec4}

In the previous section, we have seen that the sum $s(\ell)$ of particular cyclotomic numbers actually answers Conjecture~\ref{conjA} and \ref{conjB} for some $\ell$. In this section, we will see how $s(\ell)$ connects to squares of $\zz_p$ and answer both conjectures for more $\ell$. 
First of all, let us back to the equation $1 + b = c$ in Conjecture~\ref{conjB} where $b \in t L$ and $c \in t^2 L$ for $t \in \zz_p^{\times}$ such that $\zz_p^{\times} / L = \langle t L \rangle$. In this case, $(1 + t L) \cap t^2 L$ is nonempty. Write $b = t \beta$ and $c = t^2 \gamma$ for $\beta, \gamma \in L$. It gives $t^2 \gamma - t \beta - 1 = 0$. Viewing the equation as a quadratic polynomial in $t$, the discriminant provides that $\beta^2 + 4 \gamma$ is a square. Equivalently, $1 + 4 \beta^{-2} \gamma$ is a square. Since $2 \in L$, it follows that $1 + 4 \beta^{-2} \gamma$ is a square in $L$. This leads us to think of squares in $1 + L$. In general, if $1 + \alpha$ is a square for $\alpha \in L$, namely $1 + \alpha = (s-1)^2$ for some $s \in \zz_p^{\times}$, then we have $1 + s (2 \alpha^{-1}) = s^2 \alpha^{-1}$ which gives an element of $(1 + s L) \cap s^2 L$ because $2 \alpha^{-1}, \alpha^{-1} \in L$. Note that $s L = g^i L$ for some $0 \leq i \leq \ell-1$ where $g$ is a given primitive root. It follows that $A(i, 2 i) \neq 0$. Consequently, every square of $1 + L$ (if exists) leads to a nonzero cyclotomic number $A(i, 2 i)$ for some $0 \leq i \leq \ell-1$. We will see $A(i, 2 i) \neq 0$ for some $i \neq 0$. It follows that $s(\ell) \neq 0$ when $\ell$ is prime.


Given a primitive root $g$ of $\zz_p$, we generally consider squares in $1 + g^k L$ for $0 \leq k \leq \ell-1$. 
Fix $k$ and let $r \in \zz_p$ such that $r^2 \in 1 + g^k L$. Then $r^2 - 1 \in g^k L$. Observe that $r^2 - 1 = (r-1) (r+1)$ and $r \neq \pm 1$ since $0 \not\in g^k L$. Thus, $r-1 \in g^i L$ and $r+1 \in g^j L$ for some $i, j$ satisfying $i + j \equiv k \pmod{\ell}$ as $(g^i L) (g^j L) = g^k L$. In particular, $r \in (1 + g^i L) \cap (-1 + g^j L)$. Conversely, if $r' \in (1 + g^i L) \cap (-1 + g^j L)$ and $i + j \equiv k \pmod{\ell}$, then we have ${r'}^2 \in 1 + g^k L$. Let $$R(i, j) = (1 + g^{i} L) \cap (-1+g^{j} L)$$ for integers $i, j$ and let $S_k$ be the set of squares of $1 + g^k L$. The discussion above provides a surjective map $$\begin{array}{cccc} f_k: & \displaystyle \bigcup_{\substack{0 \leq i, j \leq \ell-1, \\ i + j \equiv k \; (\text{mod } \ell)}} R(i, j) & \to & S_k \\ & r & \mapsto & r^2. \end{array}$$ 

\begin{lem}
\label{lem3.5}
Let $R(i, j)$ be as above.
\begin{enumerate}
\item[(i)]
For $0 \leq i, i', j, j' \leq \ell-1$, if either $i \neq i'$ or $j \neq j'$, then $R(i, j) \cap R(i', j') = \varnothing$. Thus, the domain of $f_k$ is a disjoint union of $R(i, j)$.
\item[(ii)]
For $i, j$ modulo $\ell$, $|R(i, j)| = A(i, j).$
\end{enumerate}
\end{lem}

\begin{proof}
For (i), when $i \neq i'$, one has $(1 + g^i L) \cap (1 + g^{i'} L) = \varnothing$ because $g^i L \cap g^{i'} L = \varnothing$. Similarly, when $j \neq j'$, one also has $(-1 + g^j L) \cap (-1 + g^{j'} L) = \varnothing$. The result follows.

For (ii), consider the map $$\begin{array}{cccc} \psi : & R(i, j) & \to & (1 + g^i L) \cap g^j L \\ & r & \mapsto & 2^{-1} (1 + r). \end{array}$$ It is easy to see that $\psi$ is injective. For every $s \in (1 + g^i L) \cap g^j L$, $-1 + 2 s \in T(i, j)$ since $2 \in L$. Moreover, $\psi(-1 + 2 s) = s$. Hence, $\psi$ is a bijection and the result follows by the definition of $A(i, j)$. 
\end{proof}

The following lemma provides a connection between a particular sum of cyclotomic numbers and the number of squares.

\begin{lem}
\label{lem3.4}
The number of squares of $1 + L$ is $$\frac{1}{2} + \frac{1}{2} \sum_{i = 0}^{\ell-1} A(i, -i).$$ For $1 \leq k \leq \ell-1$, the number of squares of $1 + g^k L$ is $$\frac{1}{2} \sum_{i = 0}^{\ell-1} A(i, k-i).$$
\end{lem}

\begin{proof}
Let $R(i, j)$, $S_k$ and $f_k$ be as above. Note that if $0 \in R(i, j)$, then $0 \in 1 + g^i L$ and $0 \in -1 + g^j L$. As $1, -1 \in L$, we deduce that $0 \in R(i, j)$ if and only if $i = j = 0$. Recall from Lemma~\ref{lem3.5}(i) that the domain of $f_k$ is a disjoint union of $R(i, j)$ for $0 \leq i, j \leq \ell-1$ and $i + j \equiv k \pmod{\ell}$. We discuss into two cases that $i = j$ and $i \neq j$.

Case 1: $i = j$. Clearly, we have $R(i, i) = - R(i, i)$. Moreover, $f_k(r) = r^2 = f_k(-r)$. Thus, the restriction of $f_k$ to $R(i,i) \setminus \{0\}$ is two-to-one. If $0 \in R(i, i)$, then $i = 0$ and this holds only for $k = 0$. So the image of $R(i, i)$ under $f_0$ has the size $\frac{1}{2} + \frac{1}{2} A(0,0)$ if $i = 0$ and the size $\frac{1}{2} A(i, i)$ if $i \neq 0$. For $k = 1, \ldots, \ell-1$, the image of $R(i, i)$ under $f_k$ has the size $\frac{1}{2} A(i, i)$.

Case 2: $i \neq j$. In this case, $R(i, j) \neq R(j, i)$ and $R(j, i) = - R(i, j)$. Thus, the restriction of $f_k$ to $R(i, j) \cup R(j, i)$ is also two-to-one. It follows that the image of $R(i, j) \cup R(j, i)$ under $f_k$ has the size $\frac{1}{2} (A(i, j) + A(j, i))$.

Since $f_k$ is surjective, Case~1 and 2 lead to $$|S_0| = \frac{1}{2} + \frac{1}{2} \sum_{\substack{0 \leq i, j \leq \ell-1, \\ i + j \equiv 0 \; (\text{mod } \ell)}} A(i, j) \quad \text{and} \quad |S_k| = \frac{1}{2} \sum_{\substack{0 \leq i, j \leq \ell-1, \\ i + j \equiv k \; (\text{mod } \ell)}} A(i, j)$$ for $k = 1, \ldots, \ell-1$. Finally, each $i$ leads to a unique $j$ such that $j \equiv k - i \pmod{\ell}$. The proof is completed.
\end{proof}

\begin{exmp}
Let $p = 31$. Then $L = \{\pm 1, \pm 2, \pm 2^2, \pm 2^3, \pm 2^4\}$ and $\ell = 3$. Choose $g = 3$ a primitive root. Then the sets $S_k$ of squares of $1 + g^k L$ are as follows: $$S_0 = \{0, 2, 5, 9, 16, 28\}, \quad S_1 = \{4, 7, 8, 18, 20, 25\}, \quad S_2 = \{10, 14, 19\}.$$ On the other hand, the cyclotomic matrix (Section~\ref{sec3}) is $$\left( \begin{matrix} A(0,0) & A(0,1) & A(0,2) \\ A(1,0) & A(1,1) & A(1,2) \\ A(2,0) & A(2,1) & A(2,2) \end{matrix} \right) = \left( \begin{matrix} 3 & 4 & 2 \\ 4 & 2 & 4 \\ 2 & 4 & 4 \end{matrix} \right).$$ Indeed, $|S_0| = \frac{1}{2} + \frac{1}{2} (A(0,0) + A(1,2) + A(2, 1))$, $|S_1| = \frac{1}{2} (A(0,1) + A(1,0) + A(2,2))$ and $|S_2| = \frac{1}{2} (A(0,2) + A(1,1) + A(2, 0))$.
\end{exmp}




Let us concentrate on $1 + L$. According to Lemma~\ref{lem4.1}, $A(i, - i) = A(-i, i) = A(i, 2 i)$. Thus, the number of squares of $1 + L$ is also $$\frac{1}{2} + \frac{1}{2} \sum_{i = 0}^{\ell-1} A(i, 2i).$$ Recall that $$s(\ell) = \sum_{\substack{1 \leq i \leq \ell, \\ (i, \ell) = 1}} A(i, 2i).$$ We have the following conclusion. 

\begin{prop}
\label{prop4.2}
If $\ell$ is an odd prime, then the number of squares of $1 + L$ is $\frac{1}{2} (1 + A(0, 0) + s(\ell))$.
\end{prop}


Note that $1 + L$ contains at least one square, says $0$. When $\ell$ is an odd prime, Proposition~\ref{prop4.2} provides that we will have $s(\ell) \neq 0$ if $A(0, 0)$ is small enough compared to the number of squares of $1 + L$. Let us see an example below.

\begin{exmp}
\label{exmp4.1}
Let $p = 331$. Then $|L| = 30$ and $\ell = 11$. One can check that $A(0,0) = |(1+L) \cap L| = 5$. Thus, $1 + L$ has $3 + \frac{s(11)}{2}$ squares by Proposition~\ref{prop4.2}. Note that $1 + L$ contains at least $4$ squares such as $0$, $1 + 4 = 233^2$, $1-4 = 63^2$ and $1+8 = 3^2$. We have $s(11) \neq 0$.  
Indeed, if we pick a primitive root $g = 3$, then one can compute $A(1, 2) = 3$. This also illustrates $s(11) \neq 0$ since it is independent on the choices of primitive roots by Lemma~\ref{prop2.1}.
\end{exmp}

We now focus on $A(0,0)$ and the number of squares of $1 + L$. In \cite{BetHirKomMun}, the authors study upper bounds of cyclotomic numbers. They determine $A(i, j)$ by computing the rank of certain matrix with entries in $\zz_p$. Estimating the rank provides an upper bound of $A(i, j)$ for general $\ell$. We use their result for $A(0, 0)$ below.

\begin{lem}[{\cite[Theorem~1.1(iii)]{BetHirKomMun}}]
\label{lem4.2}
If $\ell \geq 3$, then $$A(0, 0) \leq \left\lceil \frac{p-1}{2 \ell} \right\rceil - 1$$ where $\lceil \cdot \rceil$ denotes the smallest integer larger than or equal to a given number.
\end{lem}

In our situation, $\frac{p-1}{\ell} = |L|$ and $|L|$ is even as $-1 \in L$. Thus, Lemma~\ref{lem4.2} implies $$A(0,0) \leq \frac{|L|}{2} - 1.$$ 

We study the number of squares of $1 + L$ in a different way from cyclotomic numbers. Denote $\left( \frac{\cdot}{p} \right)$ the {\it Legendre symbol} by $$\left( \frac{x}{p} \right) = x^{\frac{p-1}{2}} = \left\{ \begin{array}{cl} 0 & \text{if } x = 0; \\ 1 & \text{if $x \neq 0$ is a square}; \\ -1 & \text{if $x \neq 0$ is not a square,} \end{array}\right.$$ for $x \in \zz_p$. Then $x$ is a square in $\zz_p$ if and only if $\left( \frac{x}{p} \right) = 0, 1$. Thus, for every $x, y \in \zz_p$, the collection $\{x, y, x y\}$ must contain at least one square because of the multiplication $$\left( \frac{x}{p} \right) \cdot \left( \frac{y}{p} \right) = \left( \frac{x y}{p} \right)$$ holds (see also \cite[Chapter~9]{Bur}). We have the following conclusion.

\begin{lem}
\label{lem4.3}
If $x, y \in \zz_p$ and $x y \neq 0$, then $$\frac{1}{2}\left(1 + \left( \frac{x}{p} \right) \right) + \frac{1}{2} \left( 1 + \left( \frac{y}{p} \right) \right) + \frac{1}{2} \left( 1 + \left( \frac{x y}{p} \right) \right) \geq 1.$$
\end{lem}

We have a lower bound of the number of squares of $1 + L$ as follows.

\begin{lem}
\label{lem4.4}
The number of squares of $1 + L$ is greater than $\frac{1}{4} |L|$.
\end{lem}

\begin{proof}
Let $m = |1+L| = |L|$ and let $\alpha$ be a generator of $L$. The integer $m$ is even because $-1 \in L$. Moreover, $\alpha^{\frac{m}{2}} = -1$. Denote $M$ the number of squares of $1 + L$. Then $$M = 1 + \sum_{\substack{i=0, \\ i \neq \frac{m}{2}}}^{m-1} \frac{1}{2} \left( 1 + \left( \frac{1+\alpha^{i}}{p} \right) \right)$$ where the starting $1$ corresponds to $1 + \alpha^{\frac{m}{2}} = 0$. Since $-1 \in L$, $1 + L = 1 - L$ and we also have $$M = 1 + \sum_{i = 1}^{m-1} \frac{1}{2} \left( 1 + \left( \frac{1-\alpha^{i}}{p} \right) \right)$$ where the starting $1$ corresponds to $1 - \alpha^{0} = 0$. 

For every $i = 0, 1, \ldots, m-1$, the collection $\{1 + \alpha^i, 1 - \alpha^i, 1 - \alpha^{2 i}\}$ must contain at least one square since $(1 + \alpha^i) (1 - \alpha^i) = 1 - \alpha^{2 i}$. By Lemma~\ref{lem4.3}, $$\sum_{\substack{i=1, \\ i \neq \frac{m}{2}}}^{m-1} \frac{1}{2} \left( 1 + \left( \frac{1+\alpha^i}{p} \right) \right) + \frac{1}{2} \left( 1 + \left( \frac{1-\alpha^{i}}{p} \right) \right) + \frac{1}{2} \left( 1 + \left( \frac{1-\alpha^{2i}}{p} \right)\right) \geq m-2.$$ The first two sums of $\frac{1}{2} \left( 1 + \left( \frac{1+\alpha^i}{p} \right) \right)$ and $\frac{1}{2} \left( 1 + \left( \frac{1-\alpha^i}{p} \right) \right)$ are partial sums of $M-1$ with nonnegative terms. It follows that $2 (M-1) + M'\geq m-2$ where $$M' = \sum_{\substack{i=1, \\ i \neq \frac{m}{2}}}^{m-1} \frac{1}{2} \left( 1 + \left( \frac{1-\alpha^{2i}}{p} \right) \right).$$ Observe that $\alpha^{2 i} = \alpha^{2 (\frac{m}{2} + i)}$ for $i = 1, \ldots, \frac{m}{2} - 1$. Thus, $$M' = 2 \sum_{i = 1}^{\frac{m}{2} - 1} \frac{1}{2} \left( 1 + \left( \frac{1-\alpha^{2i}}{p} \right) \right).$$ The set $\{1 - \alpha^{2i} \mid i = 1, \ldots, \frac{m}{2} - 1\}$ is a subset of $1 - L$ and then $\frac{M'}{2}$ is a partial sum of $M-1$. We obtain $M' \leq 2 (M-1)$. 
As a consequence, we have $4 (M-1) \geq m-2$. Hence, $M \geq \frac{m}{4} + \frac{1}{2}$, as desired.
\end{proof}

Now we can prove the main theorem below.

\begin{thm}
\label{thm3.2}
Let $p$ be an odd prime. If $\ell = [\zz_p^{\times} : \langle -1, 2 \rangle]$ is an odd prime, then $s(\ell) \neq 0$. 
\end{thm}

\begin{proof}
According to Proposition~\ref{prop4.2}, the number of squares of $1 + L$ is $\frac{1}{2} (1 + A(0,0) + s(\ell))$. By Lemma~\ref{lem4.4}, this number is greater than $\frac{1}{4} |L|$ and then $1 + A(0,0) + s(\ell) > \frac{1}{2} |L|$. From the below of Lemma~\ref{lem4.2}, $1 + A(0,0) \leq \frac{1}{2} |L|$. Therefore, $s(\ell) > 0$.
\end{proof}

Theorem~\ref{thm3.2} provides the positivity of Conjecture~\ref{conjB} for primes $p$ such that $[\zz_p^{\times} : \langle -1, 2 \rangle]$ is a prime. This also answers Conjecture~\ref{conjA} for such primes. Thus we have the following consequence.

\begin{cor}
\label{cor3.8}
Let $p$ be an odd prime such that $4 \nmid o_p(2)$. If $\ell = [\zz_p^{\times} : \langle -1, 2 \rangle]$ is an odd prime, then the optimal conflict-avoiding codes of length $p$ and weight $3$ has the size $$\frac{p-1-2 \ell}{4} + \left\lfloor \frac{\ell}{3} \right\rfloor.$$
\end{cor}


For example, if $p = 1229241823 > 2^{30}$, then $o_p(2)$ is odd and $\ell = 18307$ is a prime. Thus, we can easily conclude that the size of optimal conflict-avoiding codes of length $p$ and weight $3$ is $307307404$. Indeed, one can check that $A(1, 2) = 4$ with the primitive root $3$. One element of $(1 + 3L) \cap 9L$ is $1 + 3 \cdot 2^{1128543547} = 9 \cdot 2^{249779730}$.

An optimal CAC of length $p^k$, $k \geq 1$, can be constructed from an optimal CAC of length $p$ by \cite[Theorem~8]{FuLoShu}. Applying Corollary~\ref{cor3.8} to \cite[Theorem~8]{FuLoShu}, we obtain the following conclusion.

\begin{cor}
\label{cor4.2}
Let $p > 3$ be a non-Wieferich prime such that $4 \nmid o_p(2)$. If $\ell = [\zz_p^{\times} : \langle -1, 2 \rangle]$ is an odd prime, then an optimal conflict-avoiding code of length $p^k$ and weight $3$ has the size $$\frac{p^k - 1 - 2 k \ell}{4} + k \left\lfloor \frac{\ell}{3} \right\rfloor$$ for every $k \in \nn$.
\end{cor}

\section{Conclusion Remark}

According to Corollary~\ref{cor3.9} and \ref{cor3.8}, we now know the size of optimal CAC of prime length $p$ and weight $3$ where $4 \nmid o_p(2)$ and either $\ell = \left[\zz_p^{\times} : \langle -1, 2 \rangle \right] = 4$ or $\ell$ is an odd prime. 
For composite number $\ell \geq 6$, it is not easy to determine whether $s(\ell) \neq 0$ or not. For example, let $\ell = 6$ and using properties given in Section~\ref{sec3}, the extended cyclotomic matrix can be simplified such as $$B = \left( \begin{matrix} b_{0,0} & b_{0,1} & b_{0,2} & b_{0,3} & b_{0,4} & b_{0,5} \\ b_{0,1} & b_{0,5} & b_{1,2} & b_{1,3} & b_{1,4} & b_{1,2} \\ b_{0,2} & b_{1,2} & b_{0,4} & b_{1,4} & b_{2,4} & b_{1,3} \\ b_{0,3} & b_{1,3} & b_{1,4} & b_{0,3} & b_{1,3} & b_{1,4} \\ b_{0,4} & b_{1,4} & b_{2,4} & b_{1,3} & b_{0,2} & b_{1,2} \\ b_{0,5} & b_{1,2} & b_{1,3} & b_{1,4} & b_{1,2} & b_{0,1} \end{matrix} \right).$$ Suppose that $s(6) = b_{1,2} + b_{5, 4} = 0$. Then $$B = \left( \begin{matrix} b_{0,0} & b_{0,1} & b_{0,2} & b_{0,3} & b_{0,4} & b_{0,5} \\ b_{0,1} & b_{0,5} & 0 & b_{1,3} & b_{1,4} & 0 \\ b_{0,2} & 0 & b_{0,4} & b_{1,4} & b_{2,4} & b_{1,3} \\ b_{0,3} & b_{1,3} & b_{1,4} & b_{0,3} & b_{1,3} & b_{1,4} \\ b_{0,4} & b_{1,4} & b_{2,4} & b_{1,3} & b_{0,2} & 0 \\ b_{0,5} & 0 & b_{1,3} & b_{1,4} & 0 & b_{0,1} \end{matrix} \right).$$ By Lemma~\ref{lem3.3}(iv), we will obtain that $$\left\{ \begin{array}{l} b_{0,0} + b_{0,2} + b_{0,3} + b_{0,4} = b_{1,3} + b_{1,4}, \\ b_{0,1} + b_{0,5} = b_{0,2} + b_{0,4} + b_{2,4} = 2 b_{0,3} + b_{1,3} + b_{1,4}, \\ b_{0,3} + b_{1,3} + b_{1,4} = \frac{|L|}{2}. \end{array} \right.$$ This gives the equality $$b_{2,4} = b_{0,0} + 3 b_{0,3}.$$  
On the other hand, according to Lemma~\ref{lem3.4}, the number of squares of $1 + L$ is $$\frac{1}{2} (b_{0,0} + b_{0,3} + 2b_{2,4}).$$ By Lemma~\ref{lem4.4}, this number is greater than $\frac{1}{4} |L|$ and then $b_{0,0} + b_{0,3} + 2 b_{2,4} > \frac{1}{2} |L|$. Using the equality above, we obtain $3 b_{0,0} + 7 b_{0, 3} > \frac{1}{2} |L|$. However, the upper bound for $b_{0,0}$ given in Lemma~\ref{lem4.3} (or for $b_{0, 3}$ given in \cite{BetHirKomMun}) can not give further information in this case. For composite numbers $\ell \geq 6$, we think that finding smaller upper bounds for cyclotomic numbers may lead to a better approach.
We hope that these ideas could give some inspiration for doing this problem.

\section*{Acknowledgment}

We would like to thank Professor Yuan-Hsun Lo for bringing \cite{FuLoShu} to our attention. 
The first named author is partially supported by MOST grant 110-2115-M-003-007-MY2. 
The second named author is partially supported by MOST grant 111-2115-M-003-005. 
The third named author is supported MOST grant 111-2811-M-003-028.

\bibliographystyle{alphaurl} 
\newcommand{\etalchar}[1]{$^{#1}$}

\end{document}